\documentclass[pdflatex,sn-mathphys-num]{sn-jnl}

\usepackage{graphicx}%
\usepackage{multirow}%
\usepackage{amsmath,amssymb,amsfonts}%
\usepackage{amsthm}%
\usepackage{mathrsfs}%
\usepackage[title]{appendix}%
\usepackage{xcolor}%
\usepackage{textcomp}%
\usepackage{manyfoot}%
\usepackage{booktabs}%
\usepackage{algorithm}%
\usepackage{algorithmicx}%
\usepackage{algpseudocode}%
\usepackage{listings}%

\theoremstyle{thmstyleone}%
\newtheorem{theorem}{Theorem}
\newtheorem{lemma}[theorem]{Lemma}
\newtheorem{proposition}[theorem]{Proposition}

\theoremstyle{thmstyletwo}%

\newtheorem{example}[theorem]{Example}

\theoremstyle{thmstylethree}%
\newtheorem{definition}{Definition}%

\newcommand\seqnum[1]{\rm{#1}}
\DeclareMathOperator{\lcm}{lcm}
\def\NN{\mathbb{N}}
\def\PP{\mathbb{P}}
\def\ZZ{\mathbb{Z}}
\def\FF{\mathbb{F}}
\def\TT{\mathbb{T}}
\def\QQ{\mathbb{Q}}
\def\RR{\mathbb{R}}
\newcommand\ord{{\rm{ord}}}
\newcommand\fix{{\rm{Fix}}}
\newcommand\fixset{{\rm{FixSet}}}
\newcommand\notdivides{\mathrel{\kern-3pt\not\!\kern4.3pt\bigm|}}
\newcommand\divides\mid
\newcommand\smalldivides{\mathrel{\kern-2pt\kern3.5pt|}}
\newcommand\smallnotdivides{\mathrel{\kern-2pt\not\!\kern3.5pt|}}
\def\ge{\geqslant}
\def\le{\leqslant}
\def\eul{\rm{e}}
\DeclareMathOperator{\denom}{denom}
\DeclareMathOperator{\support}{supp}

\usepackage{tikz}
\usetikzlibrary{shapes.geometric}
\newcommand{\Ngon}[2][]{\vcenter{\hbox{\begin{tikzpicture}
\node[regular polygon,regular polygon sides=#2,draw,minimum size=1cm,#1](#2-gon){};
\foreach \X in {1,...,#2}{\fill (#2-gon.corner \X) circle[radius=2pt];}
\end{tikzpicture}}}}
\newcommand{\Twogon}{\vcenter{\hbox{\begin{tikzpicture}
\draw (0,1) -- (1,1);
  \fill (0,1) circle (2pt);
  \fill (1,1) circle (2pt);
\end{tikzpicture}}}}

\begin{document}

\title[Local structure of classical sequences,
regular primes, and dynamics]{Local structure of classical sequences,
regular primes, and dynamics}


\author*[1]{\fnm{Sawian} \sur{Jaidee}}\email{jsawia@kku.ac.th}
\equalcont{These authors contributed equally to this work.}

\author{\fnm{Patrick} \sur{Moss}}
\equalcont{These authors contributed equally to this work.}

\author[2]{\fnm{Thomas} \sur{Ward}}\email{tbward@gmail.com}
\equalcont{These authors contributed equally to this work.}

\affil*[1]{\orgdiv{Department of Mathematics}, \orgname{Khon Kaen University}, \city{Khon Kaen}, \postcode{40002}, \state{Muang District}, \country{Thailand}}

\affil[2]{\orgdiv{Department of Mathematical Sciences}, \orgname{Durham University}, \city{Durham}, \postcode{DH1 3LE}, \state{County Durham}, \country{United Kingdom}}


\abstract{We introduce the notions of local realizability
at a prime and algebraic realizability of
an integer sequence. After discussing
this notion in general we consider it for the
Euler numbers, the Bernoulli denominators, and
the Bernoulli numerators.
This gives, in particular, a dynamical characterisation
of the Bernoulli regular primes.
Algebraic realizability of the Bernoulli
denominators is shown at every prime,
giving a different perspective on the
arithmetic structure of this sequence.
We show that the sequence of Euler
numbers cannot be realized on a nilpotent
group, marking a fundamental difference
between the nature of the Bernoulli and Euler
sequences.}




\maketitle

\section{Introduction}

We are concerned here with `local' properties of integer
sequences, meaning the~$p$-part of the sequence for
a given prime~$p$. This notion is of particular relevance
for the question of `algebraic realizability'---does the
sequence count periodic points for some group endomorphism?
We will discuss this kind of question in general and
then specifically in relation to sequences derived from
the classical Bernoulli and Euler numbers.

Some parts of the results here---in particular, the
`global' realizability
of the Euler sequence and the Bernoulli denominators---have
been proved several times, in particular by
Arias de Reyna~\cite{zbMATH02203410}. Where our
arguments are different they are outlined here
nonetheless, in part because in some cases they
expose different aspects of these sequences.
Two striking examples are that the elegant argument
of Arias de Reyna also shows a quite unexpected
`shift invariance' property;
on the other hand our arguments explain
some (but not all) of the `local' and the `algebraic'
properties of these sequences.
To some extent we wish to use the argument
below for algebraic realizability of the Bernoulli
denominator sequence as an advertisment
for the idea of studying integer sequences
locally.

As with any arguments involving the Euler and Bernoulli
numbers it is likely that we have inadvertantly reinvented congruences
and identities that are well-known to somebody, and
missed simplifications flowing from congruences
we are not aware of. We have done our
best to appropriately calibrate which congruences
might be viewed as `standard' and
which require justification.
Similarly, some of the elementary group theory
results we need are undoubtedly well-known in
the right quarters.

\section{Notation}

Sequences will begin with
index~$1$ (not~$0$)
and are denoted
\[
a=(a_n)=(a_n)_{n\ge1}=(a_n)_{n\in\NN}.
\]
Here~$\NN$ denotes the natural
numbers~$\{1,2,3,\dots\}$.
We will always identify the
first few terms of a sequence explicitly
for two reasons:
Firstly, (almost) all the properties we will be dealing with are
{\it{a priori}} acutely sensitive to shifting the
sequence;
secondly,
where sequences are identified with an entry (denoted~A
followed by a number)
in the Online Encyclopedia of Integer Sequences~\cite{OEIS}
the claimed equality only means identity up
to a possible shift.

We denote the rational primes~$\{2,3,5,\dots\}$
by~$\mathbb{P}$ and for~$p\in\mathbb{P}$ the~$p$-part of an integer~$n$
is written~$\lfloor n\rfloor_p=p^{\ord_p(n)}$, where~$\ord_p(n)$
is the number of times that~$p$ divides into~$n$.
For a sequence~$a$ we then write~$\lfloor a\rfloor_p$ for the
sequence of~$p$-parts~$(\lfloor a_n\rfloor_p)$.

A `dynamical system' here means
a map~$T\colon X\to X$ on a set with the property
that
\[
\fixset_n(T)
=
\{x\in X\mid T^nx=x\}
\]
is a finite set for all~$n\ge1$, where~$T^n$ denotes the~$n$th
iterate of~$T$. We write~$\fix_n(T)$ for the
cardinality~$\vert\fixset_n(T)\vert$.
We also as usual write~$\mu$ for the
M{\"o}bius function and~$\phi$
for the Euler function.

\begin{definition}
A sequence~$a=(a_n)_{n\ge1}$ of non-negative integers
is said to be
\begin{itemize}
\item \emph{(globally) realizable} if there is a dynamical
system~$(X,T)$ that realizes it, meaning that~$a_n=\fix_n(T)$ for all~$n\ge1$,
\item \emph{realizable at a prime~$q$} if~$(\lfloor a_n\rfloor_q)_{n\ge1}$
is a realizable sequence, and
\item \emph{(property) algebraically} realizable if there is a
group~$X$ (with that property) and an endomorphism~$T\colon X\to X$
with~$a_n=\fix_n(T)$ for all~$n\ge1$.
\end{itemize}
\end{definition}

Thus we can speak about `abelian, nilpotent,\dots realizable'
and drop the word `algebraically' where it is clear from
the context. It would be natural to call realizability
at a prime `local' realizability, but we will also make
use of local properties of groups in the sense of finite
generation and so avoid this to minimize confusion.
Thus the concept that any number theorist might
wish to call `locally realizable at~$p$' will simply be
`realizable at~$p$'.

The topological analogue of algebraic
realizability---imposing additional
continuity or smoothness conditions
on the system witnessing realizability---turns out not
to matter in the following sense. Windsor~\cite{MR2422026}
showed that if a sequence is realizable in the abstract
combinatorial sense above then the
map realizing it may be taken to be a~$C^{\infty}$ diffeomorphism
of the annulus. Constraining the space on which
the realizing map acts brings
in different questions.
Realizability by a continuous map on a manifold
entails additional conditions arising from the
Lefschetz fixed-point theorem, but a full
picture is not yet available.
For example, Dorota Chanko in Gda{\'n}sk
has recently shown that
realizability by a
continuous map on the~$2$-sphere
requires abstract realizability,
the expected combinatorial conditions
arising directly from the
Lefschetz fixed-point theorem on the~$2$-sphere,
and further non-trivial combinatorial conditions.

Clearly the term-by-term product of realizable
(or algebraically realizable) sequences
gives a realizable (respectively, an algebraically realizable) sequence, so
if~$a$ is realizable at~$q$ for all~$q\in\mathbb{P}$
then~$a$ is realizable.
This may be viewed as a simple `local-to-global' principle.
In our context it is the `global-to-local' direction that
brings in the complexities.
The extent to which the global-to-local converse
may or may not hold involves subtle information about the
arithmetic of the sequence involved.

Before proceeding it may be helpful to say a few words
about how these properties may be proved for a given
sequence.
To show that a sequence~$a$ of non-negative integers
is realizable there are two possible approaches.
The first is serendipity: Have prior knowledge of a map~$T$
with~$\fix_n(T)=a_n$ for all~$n\ge1$.
The second is brute force: Prove directly
that~$\sum_{d\smalldivides n}\mu\bigl(\frac{n}{d}\bigr)a_d$
is non-negative (the `sign condition') and divisible by~$n$
(the `Dold congruence') for all~$n\ge1$.
Proving the Dold congruence is made considerably
easier by a result of Arias de Reyna~\cite[Th.~3]{zbMATH02203410}:
It is enough to show that for every~$p\in\PP$
with~$\gcd(n,p)=1$ and~$m\ge1$ we have~$a_{np^m}\equiv a_{np^{m-1}}$
modulo~$p^m$.
To show that a sequence~$a$ is algebraically realizable
there seems to be only one approach, and that is to
construct a group endomorphism that realizes~$a$.
That is, we are not aware of a purely combinatorial
characterization in the spirit of the Dold
and sign conditions for the property of being algebraically
realizable.

The reverse direction is more straightforward: To show that
a sequence is \emph{not} realizable (at~$p\in\PP$) is simply a matter
of finding a witness to that fact. That is, some~$n$ (and~$p\in\PP$)
with the property that the~$n$th term
(or the~$p$-part of the~$n$th term)
of the sequence exhibits a failure to satisfy the Dold congruence or the sign condition.
To show that a realizable sequence is not algebraically realizable
is more difficult because of the lack of a
characterization of the property
and the great diversity of group automorphisms
(for example, the automorphisms of one-dimensional
solenoids with given entropy already exhibit
many uncountable families of sequences
with differing growth characteristics;
see~\cite{MR3033948} and~\cite{MR1458718}).

\section{Algebraic Realizability}

Notice that if a sequence~$a$ is algebraically
realized by a group endomorphism~$\theta$
on a group~$G$
then~$\fixset_n(\theta)$ is a finite~$\theta$-invariant
subgroup of~$G$ for each~$n\ge1$. It follows that
any algebraically realizable sequence is a
\emph{divisibility} sequence in the sense that~$m\divides n$
implies~$\fix_m(\theta)\divides\fix_n(\theta)$
by Lagrange's theorem.

\begin{lemma}\label{lemmaalgebraicreduction}
If~$a$ is algebraically realizable then it can
be realized by an automorphism of a countable
locally finite group. If a bounded sequence~$a$
is algebraically realizable then it can
be realized by an automorphism of a group
with~$\max\{a_n\mid n\ge1\}$ elements.
\end{lemma}

\begin{proof}
Assume first that~$a$ is realized by an endomorphism~$\theta\colon G\to G$
of a group~$G$
and let~$x$ and~$y$ be elements of~$X=\bigcup_{n\ge1}\fixset_n(\theta)$
which is a countable set as it is a countable union of
finite sets.
Then there are~$m,n\ge1$ with~$\theta^mx=x$ and~$\theta^ny=y$
so~$x,y\in\fixset_{\lcm(m,n)}(\theta)$ and hence~$xy,x^{-1}\in\fixset_{\lcm(m,n)}(\theta)$.
It follows that~$X$ is a subgroup of~$G$
and is~$\theta$-invariant since each~$\fixset_n(\theta)$
is~$\theta$-invariant. Clearly~$\theta\vert_{X}\colon X\to X$
algebraically realizes~$a$ and~$\theta\vert_{X}$
is surjective.
Let~$x_j\in\fixset_{n_j}(\theta\vert_X)$ for~$j=1,\dots,k$
so~the group they generate~$\langle x_1,\dots,x_k\rangle$
is a subgroup of the finite group~$\fixset_{\lcm(n_1,\dots,n_k)}(\theta\vert_X)$
and hence is itself finite.
It follows that~$X$ is a locally finite group.
If~$x\in X$ has~$\theta(x)=e$
(the identity element of~$G$)
then~$x\in\fixset_n(\theta)$ for some~$n\ge1$
and hence~$x=\theta^nx=\theta^{n-1}e=e$,
so~$\theta\vert_X$ is injective, showing the
first statement.

If~$a$ is bounded choose~$m\ge1$ to have
$a_m=\max\{a_n\mid n\in\NN\}$.
If for some~$n\ge1$ the set~$\fixset_n(\theta)$
is not a subgroup of~$\fixset_m(\theta)$
then
\[
\fixset_m(\theta)\subsetneq\fixset_m(\theta)\cup\fixset_n(\theta)
\]
so~$\fixset_m(\theta)\subsetneq\fixset_{\lcm(m,n)}(\theta)$
and hence~$a_m<a_{\lcm(m,n)}$, contradicting the
choice of~$m$. It follows that~$X=\fixset_m(\theta)$,
\end{proof}

The extent of our ignorance about how to characterize the
algebraic realizability
property may be seen in the following simple example.

\begin{example}
The periodic linear recurrence
\[
a=(1,1,1,1,6,1,1,1,1,6,\dots)=\seqnum{A010122}
\]
is a divisibility sequence that is realizable but not algebraically
realizable.
To see the first property, notice that the sequence is realized
by the permutation~$T=(12345)(6)$ on the set~$X=\{1,2,3,4,5,6\}$.
Lemma~\ref{lemmaalgebraicreduction} shows that if the
sequence is algebraically realizable then there is
a group endomorphism of a group with six
elements
that realizes~$a$. A na{\"i}ve (but on this
occasion successful) approach is to simply list the
automorphisms of the two possibilities~$\ZZ/6\ZZ$ and~$S_3$
and verify that none of them realize~$a$.
\end{example}

The next example shows that the class of sequences realizable
by nilpotent group endomorphisms is strictly larger than the
class of abelian realizable sequences.

\begin{example}\label{exampledihedral}
Let~$\theta\colon D_8\to D_8$ be the outer
automorphism of the dihedral group
\[
D_8=\langle u,v\mid u^4=v^2=e,vuv^{-1}=u^{-1}\rangle
\]
defined by~$\theta(u)=u$ and~$\theta(v)=uv$.
Then~$\theta$ algebraically realizes the
sequence~$a=(4,4,4,8,4,4,4,8,\dots)=\seqnum{A226276}$.
If there is an endomorphism~$\theta_a\colon G\to G$
of an abelian group that realizes~$a$ then by
Lemma~\ref{lemmaalgebraicreduction} we may assume
that~$G$ has eight elements and~$\theta$ is
an automorphism of order~$4$ with~$\fix_1(\theta)=4$
and so~$G/\fixset_1(\theta)$ has order~$2$.
For any~$x\in G\setminus\fixset_1(\theta)$
we must therefore have~$2x\in\fixset_1(\theta)$
and
\[
G\setminus\fixset_1(\theta)=\{x+g\mid g\in\fixset_1(\theta)\}
\]
since~$\fix_1(\theta)=4$.
The orbit~$O(x)=\{x,\theta(x),\theta^2(x),\theta^3(x)\}$
of~$x$ under~$\theta$
has no more than four elements since~$\theta^4=I$. We
claim that~$\vert O(x)\vert=4$
and~$O(x)\cap\fixset_1(\theta)=\varnothing$
by the following argument:
If~$\theta^nx\in\fixset_1(\theta)$
for some~$n$ with~$0<n\le3$ then~$\theta^{n+1}(x)=\theta^n(x)$
and so~$\theta(x)=x$ since~$\theta$ is an automorphism,
contradicting the assumption that~$x\notin\fixset_1(\theta)$.
If~$\theta^m(x)=\theta^n(x)$ with~$0\le m<n\le3$
then~$0<n-m\le3$
and hence~$\theta^{n-m}(x)=x$,
also a contradiction.
It follows that~$O(x)$ comprises the four
elements of~$G\setminus\fixset_1(\theta)$,
so~$\theta(x)=x+g$ for some~$g\in\fixset_1(\theta)$.
Since~$2x\in\fixset_1(\theta)$
it follows that~$2x=\theta(2x)=2x+2g$
and so~$2g=0$. On the other hand~$x\neq\theta^2(x)$
and so~$x\neq\theta^2(x)=\theta(x+g)=x+2g=x$,
a contradiction. That is, the sequence~$a$
is algebraically realizable but is not realizable
by an endomorphism of an abelian group.
\end{example}

The main result in this direction shows that
nilpotent realizability is extremely well-behaved
locally.

\begin{theorem}\label{theoremNilpotent}
A sequence is nilpotently
realizable if and only if it is locally
nilpotently
realizable at every prime.
\end{theorem}

An integer sequence~$a$ is called a~$p$-sequence
for a prime~$p$ if~$a_n=p^{\ord_p(a_n)}$ for all~$n\ge1$,
so Example~\ref{exampledihedral} is a~$2$-sequence.
We start by showing a local-to-global principle
for nilpotent realizability.

\begin{lemma}\label{lemmaLocallyNilpotent1}
If a~$p$-sequence is algebraically realizable
by a group endomorphism~$\theta\colon G\to G$
and~$G=\bigcup_{n\ge1}\fixset_n(\theta)$ then~$G$
is a locally finite~$p$-group and so is locally
nilpotent.
\end{lemma}

\begin{proof}
The group~$G$ is locally finite by Lemma~\ref{lemmaalgebraicreduction},
so suppose that
\[
x\in\fixset_n(\theta)\subset G
\]
has order~$o(x)$ not a power
of~$p$. Since~$G$ is locally finite, there is a prime~$q\neq p$
such that~$q\divides o(x)$. Since~$x\in\fixset_n(\theta)$
we must have~$q\divides\fix_n(\theta)$, contradicting the
assumption that~$(\fix_n(\theta))$ is a~$p$-sequence
and~$G$ is a~$p$-group.
It follows that~$G$ is locally nilpotent since
any finitely generated subgroup of~$G$ is
finite and hence nilpotent.
\end{proof}

\begin{lemma}\label{lemmaLocallyNilpotent2}
Assume that~$G$ is a Cartesian product~$\prod_{p\in\PP}G_p$
where each~$G_p$ is a locally finite~$p$-group and~$X<G$
is locally finite. Then~$X$ is locally nilpotent.
\end{lemma}

\begin{proof}
Let~$x=(x_p)\in X<\prod_{p\in\PP}G_p$ and let~$k=o(x)$
which is finite since~$X$ is locally finite.
Then~$o(x_p)\divides k$ and~$o(x_p)$
is a power of~$p$ for all~$p\in\PP$,
so~$x_p=1$ for all but finitely many primes~$p$.
Call the finite set~$\support(x)=\{x_p\mid(x_p)\in X\text{ and }o(x_p)>1\}$
the support of~$x$.
Given a finite subset~$\{x^{(1)},x^{(2)},\dots,x^{(m)}\}\subset X$
let
\[
S=S(\{x^{(1)},x^{(2)},\dots,x^{(m)}\})
=
\bigcup_{r=1}^{m}\support\bigl(x^{(r)}\bigr)
\]
and let
\[
X_p
=
\begin{cases}
\langle S\cap G_p\rangle&\text{if }S\cap G_p\neq\varnothing\\
\{1\}&\text{if not.}
\end{cases}
\]
Since~$S$ is finite and~$G_p$ is a locally
finite~$p$-group, the subgroup~$X_p<G_p$
is nilpotent. Moreover~$\langle x^{(1)},x^{(2)},\dots,x^{(m)}\rangle
=\prod_{p\in\PP}X_p$ and this product
has only a finite number of non-trivial groups.
It follows
that~$\langle x^{(1)},x^{(2)},\dots,x^{(m)}\rangle$ is nilpotent
and hence~$X$ is locally nilpotent.
\end{proof}

The next result proves one direction of Theorem~\ref{theoremNilpotent}.

\begin{proposition}\label{propositionLocalToGlobalForNilpotent}
If~$a$ is locally nilpotently realizable at every
prime then~$a$ is locally nilpotently realizable.
\end{proposition}

\begin{proof}
The assumption means that at each~$p\in\PP$
there is a locally nilpotent group~$X_p$
with an endomorphism~$\theta_p\colon X_p\to X_p$
so that~$\fix_n(\theta_p)=\lfloor a_n\rfloor_p$ for
all~$n\ge1$.
As in the proof of Lemma~\ref{lemmaalgebraicreduction}
we may assume without loss of generality
that~$X_p=\bigcup_{n\ge1}\fixset_n(\theta_p)$.
Since~$\theta_p$ realizes a~$p$-sequence,~$X_p$
is a locally finite~$p$-group by Lemma~\ref{lemmaLocallyNilpotent1}.
The original sequence~$a$ is then realized by
the product~$\theta=\prod_{p\in\PP}\theta_p$
on~$G=\prod_{p\in\PP}X_p$.
By Lemma~\ref{lemmaalgebraicreduction}
we have that
\[
X=\bigcup_{n\ge1}\fixset_n(\theta)
\]
is a locally finite subgroup of~$G$
so Lemma~\ref{lemmaLocallyNilpotent2}
implies that~$X$ is locally nilpotent
and so~$a$ is nilpotently realizable
by Lemma~\ref{lemmaalgebraicreduction}.
\end{proof}

The reverse direction of Theorem~\ref{theoremNilpotent}
requires some further group-theoretic ideas.

\begin{lemma}\label{lemmaSylow1}
If a group~$G$ has a unique Sylow~$p$-subgroup~$P$
then any subgroup~$H\le G$ has a unique
Sylow~$p$-subgroup and it is~$P\cap H$.
\end{lemma}

\begin{proof}
We may assume that~$G$ is not a~$p$-group.
If~$K\le G$ is a~$p$-subgroup then~$P\le PK\le G$
since~$P$ is normal in~$G$. Now~$PK$ is
a~$p$-subgroup so by maximality of
Sylow~$p$-subgroups and our assumption that~$G$ is
not a~$p$-group, we have~$PK=P$ and so~$K\le P$.
Any Sylow~$p$-subgroup of~$H\le G$ is
contained in~$P$ by this argument,
so~$H\cap P$ is the unique Sylow~$p$-subgroup
of~$H$.
\end{proof}

\begin{lemma}\label{lemmaSylowGivesLocallyNilpotent}
If~$a$ is algebraically realized by an
endomorphism on a group with
a unique Sylow~$p$-subgroup for some~$p\in\PP$
then~$a$ is locally nilpotently realizable at~$p$.
\end{lemma}

\begin{proof}
By Lemma~\ref{lemmaalgebraicreduction} we can
assume that~$a$ is realized by an automorphism~$\theta$
on the locally finite group~$X=\bigcup_{n\ge1}\fixset_n(\theta)$.
Since~$X$ is locally finite, any finitely generated
subgroup of its unique Sylow~$p$-subgroup~$P$
is a finite~$p$-group and so is nilpotent
and hence~$P$ is locally nilpotent.
We claim that~$\theta\vert_{P}$ realizes
the~$p$-part of~$a$.
Clearly~$\fixset_n(\theta\vert_P)\le\fixset_n(\theta)$
and so~$\fix_n(\theta\vert_P)\divides\fix_n(\theta)$
and hence~$\fix_n(\theta\vert_P)\divides\lfloor a_n\rfloor_p$
since~$\fix_n(\theta\vert_P)$
is a finite subgroup of the~$p$-group~$P$.
If there is an~$n\ge1$ with~$\fix_n(\theta\vert_P)\neq\lfloor a_n\rfloor_p$
then~$p\divides\frac{\lfloor a_n\rfloor_p}{\fix_n(\theta\vert_P)}$
and so~$p\divides[\fixset_n(\theta):\fixset_n(\theta\vert_P)]$
so~$\fixset_n(\theta\vert_P)$ is not a Sylow~$p$-subgroup
of~$\fixset_n(\theta)$ and we can use Lemma~\ref{lemmaSylow1}
to deduce that~$\fixset_n(\theta\vert_P)\neq P\cap\fixset_n(\theta)$.
This is not possible as~$\theta\vert_P$ is
simply a restriction of~$\theta$.
It follows that~$\fix_n(\theta\vert_P)=\lfloor a_n\rfloor_p$
for all~$n\ge1$ and~$\theta\vert_{P}$ locally nilpotently
realizes the~$p$-part of~$a$.
\end{proof}

\begin{lemma}\label{lemmaSylowTrick2}
If~$G$ is a locally finite, locally nilpotent group
then for any~$p\in\PP$ the group~$G$ has a unique
Sylow~$p$-subgroup comprising all elements of order
a power of~$p$.
\end{lemma}

\begin{proof}
Let~$\Omega$ index the set of finite~$p$-subgroups of~$G$
in the sense that
\[
\mathcal{P}
=
\{P_\omega\mid P_{\omega}\text{ is a finite~$p$-subgroup of $G$}\}
=
\{P_\omega\mid\omega\in\Omega\}
\]
and let~$H=\bigcup_{\omega\in\Omega}P_\omega$.
If~$x,y\in H$ there are~$\omega_1,\omega_2\in\Omega$
with~$x\in P_{\omega_1}$
and~$y\in P_{\omega_2}$ and so~$x^{-1}\in P_{\omega_1}$
and hence~$x^{-1}\in H$.
If~$x$ or~$y$ is the identity then~$xy\in H$.
If not then~$o(x)=p^r$ and~$o(y)=p^s$
for some~$r,s>0$.
It follows that~$K=\langle x,y\rangle$
is a finite nilpotent subgroup of~$G$ with
a unique Sylow~$p$-subgroup~$L$.
We must have~$x,y\in L$ and hence~$K=L$.
Since~$L\in\mathcal{P}$ it follows that~$xy\in H$
and so~$H$ is a subgroup of~$G$.
If~$x\in G$ has order a power of~$p$ then~$\langle x\rangle\in\mathcal{P}$
so for any endomorphism~$\theta$ of~$G$ we have~$\theta(x)\in H$,
completing the proof.
\end{proof}

We are now in a position to prove Theorem~\ref{theoremNilpotent}.
Assume that~$a$ is a sequence locally nilpotently
realized by~$\theta\colon X\to X$.
By Lemma~\ref{lemmaalgebraicreduction} we may
assume that~$X$ is locally finite and~$\theta$ is
an automorphism.
For~$p\in\PP$ Lemma~\ref{lemmaSylowTrick2} gives
a unique Sylow~$p$-subgroup and so~$a$
is locally nilpotently realized at~$p$
by Lemma~\ref{lemmaSylowGivesLocallyNilpotent}.
This holds for all primes~$p$ and so
with Proposition~\ref{propositionLocalToGlobalForNilpotent}
completes the proof of Theorem~\ref{theoremNilpotent}.

Theorem~\ref{theoremNilpotent} gives a ready
supply of realizable~$p$-sequences.
To see how this looks in an abelian setting,
we describe the periodic points of a specific
automorphism of the~$3$-torus~$\TT^3=\RR^3/\ZZ^3$ at various primes.

\begin{example}\label{exampletoral}
Let~$\theta\colon\TT^3\to\TT^3$ be the automorphism
defined by
\[
\theta\begin{pmatrix}
x\\
y\\
z
\end{pmatrix}
=
\begin{pmatrix}
0&1&0\\
0&0&1\\
1&1&0
\end{pmatrix}
\begin{pmatrix}
x\\
y\\
z
\end{pmatrix}
=
\begin{pmatrix}
y\\
z\\
x+y
\end{pmatrix}.
\]
The resulting sequence~$a=(\fix_n(\theta))$ is
the `Lehmer--Pierce' sequence associated to
the characteristic polynomial~$x^3-x-1$
(see~\cite{MR1700272} for more on these).
A calculation shows that
\[
a
=
(1,1,1,5,1,7,8,5,19,11,23,35,27,64,61,85,137,\dots)
=
\seqnum{A001945}
\]
is a linear
recurrence sequence satisfying
\[
a_{n+6}=-a_{n+5}+a_{n+4}+3a_{n+3}+a_{n+2}-a_{n+1}-a_n
\]
for~$n\ge1$.
It is an example of a `slowly growing' linear
recurrence in the sense of Lehmer~\cite{MR1503118}
and for that reason Hall~\cite{MR1574113} factorized the first~$100$
terms. Theorem~\ref{theoremNilpotent} implies
that each of the sequences~$\lfloor a\rfloor_p$
for~$p\in\PP$ is algebraically realizable.
This `local' structure is rather intricate.
At the prime~$2$ we have
\[
\lfloor a_n\rfloor_2
=
\begin{cases}
2^{3(1+\ord_2(n))}&\text{if }7\divides n,\\
1&\text{if }7\notdivides n.
\end{cases}
\]
At the prime~$3$ we have
\[
\lfloor a_n\rfloor_3
=
\begin{cases}
3^{3(1+\ord_3(n))}&\text{if }13\divides n,\\
1&\text{if }13\notdivides n.
\end{cases}
\]
At the prime~$5$ the structure is more
involved, and we have~$\lfloor a_n\rfloor_5=b_nc_n$
where
\begin{align*}
b_n
&=
\begin{cases}
5^{1+\ord_5(n)}&\text{if }4\divides n,\\
1&\text{if }4\notdivides n
\end{cases}
\intertext{and}
c_n
&=
\begin{cases}
5^{2(1+\ord_5(n))}&\text{if }24\divides n,\\
1&\text{if }24\notdivides n.
\end{cases}
\end{align*}
\end{example}

The~$p$-parts of algebraically realizable sequences
have been (implicitly) studied from the point of view
of their rate of growth (rather than their
arithmetic) in a different setting.
The~$S$-integer dynamical systems introduced
by Chothi {\it{et al.}}~\cite{MR1461206}
may be thought of in zero characteristic as
generalizations of the following set-up: Given a set~$S\subset\PP$
write~$\lfloor a\rfloor_S=\prod_{p\in S}p^{\ord_p(a_)}$
for the~$S$-part of~$a$. Then given a starting system~$\theta\colon\TT^3\to\TT^3$
like Example~\ref{exampletoral} it constructs for each~$S\subset\PP$
an algebraic dynamical system~$\theta_S\colon X\to X$
with~$\fix_n(\theta_S)=\fix_n(\theta)/\lfloor\fix_n(\theta)\rfloor_S$
for all~$n\ge1$.
Various growth rate questions for these systems
have been studied by Everest {\it{et al.}} for~$S$ finite~\cite{MR2550149}
and for~$\PP\setminus S$ finite~\cite{MR2339472}.
Analytic properties including a
conjectured P\'{o}lya--Carslon dichotomy for associated
generating functions have been studied by
Bell {\it{et al.}}~\cite{MR3217030,MR4586814}
and
Baril Boudreau {\it{et al.}}~\cite{zbMATH08045547}.

\section{Bernoulli and Euler Numbers}

Recall that the Euler numbers~$E=(E_n)$ may be defined
by
\[
\frac{2}{{\rm{e}}^t+{\rm{e}}^{-t}}=\sum_{n=0}^{\infty}E_n\frac{t^n}{n!}
\]
so~$E_n=0$ for odd~$n$. We define a non-negative
integer sequence~$e$ by
setting~$e_n$ to be~$(-1)^nE_{2n}$ for~$n\ge1$,
giving
\begin{equation}\label{equationShiftingEulerSequence}
e=(1,5,61,1385,\dots)
=
\seqnum{A000364}.
\end{equation}

The Bernoulli numbers~$B=(B_n)$ may
be defined by
\[
\frac{t}{{\rm{e}}^t-1}=\sum_{n=0}^{\infty}B_n\frac{t^n}{n!}
\]
and we write
\[
\left\vert\frac{B_{2n}}{2n}\right\vert
=
\frac{t_n}{b_n}
\]
for~$n\ge1$ with~$\gcd(t_n,b_n)=1$ for all~$n\ge1$.
It is well-known that~$b_n$ is even and~$t_n$ is odd for all~$n\ge1$.
This defines non-negative integer
sequences~$b$ and~$t$ as follows:
\begin{align*}
b&=(12,120,252,240,\dots)=\seqnum{A006953},\\
\intertext{and}
t&=(1,1,1,1,1,691,\dots)=\vert\seqnum{A001067}\vert\text{ (that is, equality up to sign)}.
\end{align*}

\begin{definition}[Kummer's Bernoulli regular primes]\label{definitionBernoulliRegularPrimes}
A prime~$q$ is said to be \emph{(Bernoulli) regular}
if it does not divide any of~$t_1,\dots,t_{(q-3)/2}$
and \emph{(Bernoulli) irregular} if it is not regular.
\end{definition}

Thus the sequence of regular primes begins
\[
(3, 5, 7, 11, 13, 17, 19, 23,\dots)=\seqnum{A007703}
\]
and the sequence of irregular primes begins
\[
(37, 59, 67, 101, 103, 131, 149,\dots)=\seqnum{A000928}.
\]
It is well-known that there are infinitely many
irregular primes, but only conjectured that there
are infinitely many regular primes. For simplicity
we will also call the anomalous prime~$2$ regular.

Several authors have proved that
various of these sequences
defined above, and many
generalizations of them, are realizable (see Zhi-Hong Sun~\cite{zbMATH06015648},
Juan Arias de Reyna~\cite{zbMATH02203410},
Patrick Moss~\cite{pm}).
We are however interested in the extent to which these
global results translate into local realizability
at primes.

\begin{theorem}\label{theoremBTEresult}
The Bernoulli and Euler sequences have the following properties:\newline
\noindent{\rm(a)} The Bernoulli denominator sequence~$b$ is
algebraically realizable at every prime.\newline
\noindent{\rm(b)} The Bernoulli numerator sequence~$t$ is realizable and is
realizable at a prime~$q$ if and only if~$q$ is a regular prime.\newline
\noindent{\rm(c)} The Euler sequence~$e$ is realizable but is not
realizable at all primes.
\end{theorem}

\section{Bernoulli Denominators}

We start by identifying a general family of sequences with
controlled algebraic realizability properties.
In order to do this we will need to be able to
construct integer matrices with prescribed properties as
follows.

\begin{lemma}\label{lemmaPatMatrixTrick}
Let~$q=p^m$ with~$m\in\NN$ and~$p\in\PP$.
Then there exists an~$m\times m$ integer matrix~$A\in M_{m,m}(\ZZ)$
satisfying
\begin{enumerate}
\item[\rm(1)] $\det(A^n-I)\not\equiv0$ modulo~$p$ if~$q-1\notdivides n$ and
\item[\rm(2)] $A^{q-1}=I+pB$, where~$B\in M_{m,m}(\ZZ)$
and~$\det(B)\neq0$ modulo~$p$.
\end{enumerate}
\end{lemma}

\begin{proof}
As usual we write~$\FF_q$ for the field of~$q$ elements
with its prime field~$\FF_p$.
The multiplicative group~$\FF_q^{\times}=(\FF_q\setminus\{0\},\times)$
is cyclic, so contains an element~$g\in\FF_q^{\times}$
of order~$q-1$. Choosing a basis for~$\FF_q$ viewed
as a vector space over the ground field~$\FF_p$ and expressing
the map~$x\mapsto gx$ on~$\FF_q$ gives a matrix~$A\in M_{m,m}(\FF_p)$
such that~$\det(A^n-I)\neq0$ if~$q-1\notdivides n$
and~$A^{q-1}=I$.
By using the modular representatives~$\{0,1,\dots,p-1\}\subset\ZZ$
we can think of~$A$ as an element of~$M_{m,m}(\ZZ)$.
When~$A$ is viewed as an integer matrix we have~$A^{q-1}=I+pB$
for some matrix~$B\in M_{m,m}(\ZZ)$, and we claim that it
is possible to choose~$A$ so as to ensure that~$p\notdivides\det(B)$.
By multiplying on the left and on the right we have
\[
A+pAB=A^q=A+pBA
\]
so~$A$ and~$B$ commute in~$M_{m,m}(\ZZ)$
and hence~$A(I+AB)=(I+AB)A$.
Let~$A'=A+p(I+AB)$, so~$A'\equiv A$ modulo~$p$
and~$A'$ satisfies~(1).
On the other hand we have
\begin{align*}
(A')^{q-1}
&=
A^{q-1}+p(q-1)A^{q-2}(I+AB)+p^2L\\
&=
I+pB+pqA^{q-2}+pq(I+pB)B-pA^{q-2}-p(I+pB)B+p^2L\\
&=
I-pA^{q-2}+p^2\bigl(L-\tfrac{q}{p}A^{q-2}+\tfrac{q}{p}A^{q-1}B-B^2\bigr)\\
&=
I-pA^{q-2}+p^2K
\end{align*}
for some matrices~$K,L\in M_{m,m}(\ZZ)$.

Let~$B'=-A^{q-2}+pK$ so that~$(A')^{q-1}=I+pB'$ and
\[
\det(B')
\equiv\det(-A^{q-2})\not\equiv0\pmod{p}.
\]
Thus replacing~$A$ with~$A'$ and~$B$ with~$B'$ gives the lemma.
\end{proof}

\begin{definition}
For~$k,m\in\NN$ and~$p\in\PP$ with~$p\notdivides k$,
define an integer sequence~$\ell^{(k,m,p)}$ by
\[
\ell_n^{(k,m,p)}
=
\left\{
\begin{array}{ll}
p^{m(1+\ord_p(n))}&\mbox{if }k\divides n\\
1&\mbox{if }k\notdivides n
\end{array}
\right.
\]
for~$n\ge1$.
\end{definition}

Thus, for example,~$\ell^{(1,1,2)}=(2, 4, 2, 8, 2, 4, 2, 16,\dots)=\seqnum{A171977}$.

\begin{lemma}\label{lemmaOddPrimeAlgRealizableCase}
If~$p$ is an odd prime then~$\ell^{(k,m,p)}$ is
algebraically realizable if and only if~$k$
divides~$p^m-1$.
\end{lemma}

\begin{proof}
If~$\ell^{(k,m,p)}$ is realizable, then the number
of points living on closed orbits of length~$k$
is given by
\begin{align*}
\sum_{d\smalldivides k}\mu\bigl(\tfrac{k}{d}\bigr)\ell_d^{(k,m,p)}
&=
\ell^{(k,m,p)}_{k}+\sum_{d\smalldivides k\atop{d<k}}\mu\bigl(\tfrac{k}{d}\bigr)\underbrace{\ell_d^{(k,m,p)}}_{=1\text{ as }k\smallnotdivides d}\\
&=
\ell^{(k,m,p)}_{k}+\sum_{d\smalldivides k\atop{d<k}}\mu\bigl(\tfrac{k}{d}\bigr)\\
&=
p^{m(1+\ord_p(k))}-1=p^m-1
\end{align*}
since~$p\notdivides k$. It follows that~$k\divides p^m-1$.

For the converse direction we use the
additive torsion group~$\ZZ\bigl[\tfrac1p\bigr]/\ZZ$
which we denote by
\begin{equation}\label{equationptorsiontorus}
\TT_p
=
\bigcup_{n\ge1}
\bigl\{
\tfrac{r}{p^n}\mid r\in\{0,1,\dots,p^n-1\}
\bigr\}.
\end{equation}
By Lemma~\ref{lemmaPatMatrixTrick} there is
a matrix~$A\in M_{m,m}(\ZZ)$
with~$\det(A^n-I)\not\equiv0$ modulo~$p$
if~$p^m-1\notdivides n$
and~$A^{p^m-1}=I+pB$
with~$p\notdivides\det(B)$.
Let~$c=\frac{p^m-1}{k}$.
Writing the elements of~$\TT_p^m$ as
column vectors, define an
endomorphism~$T$
of~$\TT_p$
by setting~$Tx=A^{c}x$.
We claim that~$T\colon\TT_p\to\TT_p$ algebraically
realizes the sequence~$\ell^{(k,m,p)}$.

If~$T^nx=x$ with~$k\notdivides n$
then~$A^{cn}x=x$ and so~$\det(A^{cn}-I)\not\equiv0$
modulo~$p$ since~$p^m-1\notdivides cn$.
This implies that there is a matrix~$C\in M_{m,m}(\ZZ)$
satisfying~$(I+pC)x=0$. It follows that~$x=0$ and
hence~$\fix_n(T)=1$ whenever~$k\notdivides n$.

If~$k$ divides~$n$ then we may write~$n=p^skr$
for integers~$s\ge0$ and~$r\ge1$ with~$p\notdivides r$.
Since~$A^{p^m-1}=I+pB$ we have
\[
A^{cn}
=
A^{(p^m-1)p^sr}
=
(I+pB)^{p^sr}
=
I+p^{s+1}rB+\cdots
=
I+p^{s+1}D
\]
for some~$D\in M_{m,m}(\ZZ)$ with~$\det(D)\not\equiv0$
modulo~$p$ since~$p\notdivides\det(B)$ and~$p$
is odd.
It follows that if~$T^nx=x$ then~$p^{s+1}Dx=0$
which implies that~$p^{s+1}x=0$ since~$p\notdivides\det(D)$.
Thus in this case~$\fix_n(T)=p^{m(s+1)}=p^{m(1+\ord_p(n))}$
and hence the sequence~$\ell^{(k,m,p)}$ is algebraically
realized by~$T$.
\end{proof}

As usual in relation to Bernoulli and Euler
numbers the prime~$2$ behaves somewhat differently.

\begin{lemma}\label{lemma2caseisspecial}
The sequence~$\bigl(2^{2+\ord_2(n)}\bigr)$ is
algebraically realizable.
\end{lemma}

\begin{proof}
Let~$T\colon\TT_2\to\TT_2$ be the endomorphism defined by~$T(x)=5x$ modulo~$1$
on the group~$\TT_2$ defined in~\eqref{equationptorsiontorus}.
Notice that for~$r\ge0$ and~$m\ge1$ odd
we have~$2^{r+2}\divides 5^{2^rm}-1$
and~$2^{r+3}\notdivides5^{2rm}-1$ by the binomial
theorem.
It follows that~$5^n-1=2^{\ord_2(n)+2}s$
with~$s\ge1$ odd.

If~$T^nx=x$ this gives~$x=\frac{k}{2^{\ord_2(n)+2}}$
for~$k\in\{0,\dots,2^{\ord_2(n)+2}-1\}$ showing
that~$\fix_n(T)=2^{2+\ord_2(n)}$ as desired.
\end{proof}

All the proofs of realizability for the sequences~$b$,~$t$,
and~$e$ essentially follow from one of various formulations of
the Kummer congruences.

\begin{proof}[Proof of Theorem{\rm~\ref{theoremBTEresult}(a)}]
We claim first that
\begin{equation}\label{equationbsequenceasproduct}
b_n
=
2\prod_{p\in\PP\atop{p-1\mid 2n}}p^{1+\ord_p(n)}
\end{equation}
for all~$n\ge1$.
To see this, recall that the von Staudt--Clausen theorem shows that
the denominator~$d_n$ of~$B_{2n}$ is given by
\begin{equation}\label{equationvS-Ctheorem}
d_n
=
\prod_{p\in\PP\atop{p-1\mid 2n}} p
\end{equation}
for all~$n\ge1$.
At the prime~$2$ this gives~$\lfloor d_n\rfloor_2=2$
and so~$\lfloor b_n\rfloor_2=2^{2+\ord_2(n)}$
for all~$n\ge1$.
Similarly,~$\lfloor d_n\rfloor_3=3$
and so~$\lfloor b_n\rfloor_3=3^{1+\ord_3(n)}$
for all~$n\ge1$, showing~\eqref{equationbsequenceasproduct}
for the~$2$ and~$3$-part of~$b$.
Now assume that~$p\ge5$.
If we have~$p-1\divides2n$ then~$\lfloor d_n\rfloor_p=p$
by von Staudt--Clausen, so~$B_{2n}\not\equiv0$ modulo~$p$
and hence~$\lfloor b_n\rfloor_p=p^{1+\ord_p(n)}$.
If~$p-1\notdivides2n$ then
by Adams' theorem~\cite{zbMATH02710998}
we know that~$B_{2n}\equiv0$ modulo~$\lfloor n\rfloor_p$
and so~$\lfloor b_n\rfloor_p=1$. That is,
\[
\lfloor b_n\rfloor_p
=
\left\{
\begin{array}{ll}
1&\mbox{if }p-1\notdivides 2n\\
p^{1+\ord_p(n)}&\mbox{if }p-1\divides 2n
             \end{array}\right.
\]
which shows~\eqref{equationbsequenceasproduct} for~$p$.

We next use the expression~\eqref{equationbsequenceasproduct}
to show that~$b$ is locally algebraically realizable
at every prime.
By~\eqref{equationbsequenceasproduct}, the~$2$-part
of~$b$ is the sequence~$\bigl(2^{2+\ord_2(n)}\bigr)$
which is algebraically realizable by Lemma~\ref{lemma2caseisspecial}.
For an odd prime~$p$ the same formula
shows that the~$p$-part
of~$b$ is equal to~$\ell^{(\frac{p-1}{2},1,p)}$,
which is algebraically realizable by Lemma~\ref{lemmaOddPrimeAlgRealizableCase}.
\end{proof}

\section{Bernoulli Numerators}

Before looking at the Bernoulli numerator sequence
we recall some elementary results from modular
arithmetic.

\begin{lemma}\label{lemmaStayingAlive}
Let~$n=2^rm$ with~$m$ odd,~$r\ge1$,
and~$k=\frac{n}{2}$.
Then~$\frac{5^n-1}{2^{r+2}}$
and~$\frac{5^k-1}{2^{r+1}}$ are both
odd and satisfy
\[
\frac{5^n-1}{2^{r+2}}
\equiv
\frac{5^k-1}{2^{r+1}}
\pmod{2^r}.
\]
\end{lemma}

\begin{proof}
The binomial expansion of~$(1+2^2)^{2^rm}$
shows that the two fractions are odd integers.
Then a calculation shows that
\[
\frac{5^n-1}{2^{r+2}}
-
\frac{5^k-1}{2^{r+1}}
=
2^r\Bigl(
\frac{5^k-1}{2^{r+1}}
\Bigr)^2
\equiv
0\pmod{2^r}.
\]
\end{proof}

A primitive root~$\gamma_p$ modulo~$p\in\PP$
may have the property that~$p^2\divides\gamma_p^{p-1}-1$,
which creates difficulties in many different
settings: In particular, such a primitive root
is not primitive modulo~$p^2$
(see, for example, the work of Cohen~{\it{et al.}}~\cite{MR340202}).
However, if~$\gamma$ is a primitive root modulo~$p$
but not modulo~$p^2$ then
the binomial expansion of~$(\gamma+mp)^{p-1}$ modulo~$p^2$
shows that~$\gamma+mp$ is a primitive
root modulo~$p^2$ for any~$m$ not divisible by~$p$.

\begin{lemma}\label{lemmaPrimitiveRootIdea1}
Let~$\gamma_p>1$ be a primitive root modulo an odd prime~$p\in\PP$
with~$p^2\notdivides\gamma_p^{p-1}-1$ and let~$m,r\ge1$
satisfy~$p\notdivides m$.
If~$n=\frac{p^r(p-1)m}{2}$ and~$k=\frac{n}{p}$
then~$\frac{\gamma_p^{2n}-1}{p^{r+1}}$
and~$\frac{\gamma_p^{2k}-1}{p^r}$ are
integers not divisible by~$p$
and
\[
\frac{\gamma_p^{2n}-1}{p^{r+1}}
\equiv
\frac{\gamma_p^{2k}-1}{p^r}
\pmod{p^r}.
\]
\end{lemma}

\begin{proof}
By a standard argument~$\gamma_p$ is a primitive root modulo~$p^r$
for any~$r\ge1$.
It follows that~$\gamma_p^{2n}-1=p^{r+1}c$ for some~$c\ge1$
with~$p\notdivides c$
and so~$\frac{\gamma_p^{2n}-1}{p^{r+1}}$
and~$\frac{\gamma_p^{2k}-1}{p^r}$ are
integers not divisible by~$p$.
By differentiating the geometric series
\[
\sum_{j=0}^{n}a^j=\frac{1-a^{n+1}}{1-a}
\]
for~$a\neq1$ with respect to~$a$ we obtain
\[
\sum_{j=1}^{n}ja^j
=
\frac{a\bigl(1-(n+1)a^n+na^{n+1}\bigr)}{(1-a)^2}.
\]
Using this with~$n=p-1$ and~$a=\gamma_p^{-2k}$ shows that
\begin{equation}\label{equationCongruencePrimitiveRoot1}
\frac{\gamma_p^{2n}-1}{p^{r+1}}
-
\frac{\gamma_p^{2k}-1}{p^r}
=
p^{r-1}
\Bigl(\frac{\gamma_p^{2k}-1}{p^r}\Bigr)^2
\sum_{j=1}^{p-1}j\gamma_p^{2k(p-1-j)}.
\end{equation}
Since~$p-1\divides2k$ and~$\gamma_p$ is a primitive
root modulo~$p$ we have~$\gamma_p^{2k(p-1-j)}\equiv1$
modulo~$p$ for~$j\in\{1,\dots,p-1\}$
and hence
\[
\sum_{j=1}^{p-1}j\gamma_p^{2k(p-1-s)}
\equiv
\sum_{j=1}^{p-1}j
=
\frac{p(p-1)}{2}\pmod{p}.
\]
It follows that there is an additional
factor of~$p$ in the right-hand side
of~\eqref{equationCongruencePrimitiveRoot1},
giving the lemma.
\end{proof}

In addition to Adams' theorem and the von Staudt--Clausen theorem
already used, we recall a simple form of Kummer's congruence
as follows.
If~$m,n,r$ are integers with
\[
1\le r\le 2n-1\le 2m-1
\]
then for any odd~$p\in\PP$ with~$p-1\notdivides 2n$
and~$2m\equiv2n$ modulo~$\phi(p^r)$ we
have
\begin{equation}\label{equationKummerNumerator}
\frac{B_{2m}}{2m}\equiv\frac{B_{2n}}{2n}\pmod{p^r}.
\end{equation}

We will also need a more recent formulation
of a Kummer theorem due to Young~\cite{MR1713481}:
Let~$p\in\PP$ be odd and~$n,r\ge1$
with~$r=\ord_p(n)$ and~$p-1\divides2n$.
If~$k=\frac{n}{p}$ then
\begin{equation}\label{equationYoungCongruence}
(\gamma_p^{2n}-1)\frac{B_{2n}}{2n}
\equiv
(\gamma_p^{2k}-1)\frac{B_{2k}}{2k}\pmod{p^r}
\end{equation}
where~$\gamma_p$ is a primitive
root modulo~$p$ satisfying~$p^2\notdivides\gamma_p^{p-1}-1$.
Once again the prime~$2$ behaves a little differently,
and a form of Young's theorem will be needed
for this case which requires a separate argument
following the ideas of Young.

\begin{lemma}\label{lemmaFiveCongruence1}
Let~$n$ and~$r$ be positive integers with~$r=\ord_2(n)$.
If~$k=\frac{n}{2}$ then
\[
(5^n-1)\frac{B_{2n}}{2n}
\equiv
(5^k-1)\frac{B_{2k}}{2k}\pmod{2^r}.
\]
\end{lemma}

\begin{proof}
Define a sequence of integers~$a$ by
\[
\sum_{n=0}^{\infty}a_n\frac{t^n}{n!}
=
t^{-1}\Bigl(\frac{5t}{\eul^{5t}-1}-\frac{t}{\eul^t-1}\Bigr)
=
t^{-1}\sum_{n=0}^{\infty}(5^n-1)B_n\frac{t^n}{n!}.
\]
By~\cite{MR1713481} we have
\[
a_{2n-1}
\equiv
a_{2k-1}
\pmod{2^s}
\]
where~$s=\min\{n-1,r+2\}$
and hence
\begin{equation}\label{equationYoungCongruencePrime2}
(5^{2n}-1)\frac{B_{2n}}{2n}
\equiv
(5^{2k}-1)\frac{B_{2k}}{2k}
\pmod{2^s}
\end{equation}
where~$s=\min\{n-1,r+2\}$.
Notice that~$\ord_2(5^q+1)=1$ for any~$q\ge1$,
so~\eqref{equationYoungCongruencePrime2}
implies that
\[
(5^{2n}-1)\frac{B_{2n}}{2n}
\equiv
(5^{2k}-1)\frac{B_{2k}}{2k}
\pmod{2^u}
\]
where~$u=\min\{n-2,r+1\}$.
Since~$r\le\min\{n-2,r+1\}$ for~$n>2$
this shows the lemma
in the case~$n>2$;
for~$n=2$ it is clear.
\end{proof}

\subsection{Dold Congruence for Bernoulli Numerators}

For brevity define an integer sequence~$o=(o_n)$
by
\begin{equation}\label{equationOrbitsForNumeratorSequence}
o_n=\sum_{d\smalldivides n}\mu\bigl(\tfrac{n}{d}\bigr)t_d
\end{equation}
so~$o=(1, 0, 0, 0, 0, 690, 0, 3616, 43866, \dots)$;
notice that~$(\frac{1}{n}o_n)=\seqnum{A060309}$.

\begin{lemma}\label{lemmaNumeratorCongruence}
For any~$n\ge1$ we have~$n\divides o_n$.
\end{lemma}

\begin{proof}
We can assume that~$n>1$.
Pick~$p\in\PP$ dividing~$n$ and write~$n=p^rm$ with~$r,m\ge1$
and~$p\notdivides m$.

Suppose first that~$p-1\notdivides2n$. Then~$p$
is odd so~\eqref{equationKummerNumerator}
shows that
\[
(-1)^{n+1}\frac{t_n}{b_n}
=
\frac{B_{2n}}{2n}
\equiv
\frac{B_{2k}}{2k}
=
(-1)^{k+1}\frac{t_k}{b_k}
\pmod{p^r}
\]
since~$n\equiv k$ modulo~$\phi(p^r)$
where~$k=\frac{n}{p}$.
It follows that
\begin{equation}\label{equationYoungCongruence3}
\frac{t_n}{b_n}
\equiv
\frac{t_k}{b_k}\pmod{p^r}
\end{equation}
since~$p$ being odd implies that~$(-1)^{n+1}=(-1)^{k+1}$.
By~\eqref{equationbsequenceasproduct}
we can therefore write~$b_n=hb_k$
with~$d_n\ge1$ congruent to~$1$ modulo~$p^r$.
Adams' theorem shows that~$\gcd(p,b_n)=1$
and so~\eqref{equationYoungCongruence3}
shows that
\begin{equation}\label{equationYoungCongruence4}
t_n\equiv ht_k\equiv t_k\pmod{p^r}.
\end{equation}
Since
\[
o_n=\sum_{d\divides m}\mu\bigl(\tfrac{m}{d}\bigr)\bigl(t_{p^rd}-t_{p^{r-1}d}\bigr),
\]
the congruence~\eqref{equationYoungCongruence4} shows that~$p^r\divides o_n$.

Now suppose that~$p-1\divides2n$,~$m\ge1$, and~$r=\ord_p(n)\ge1$.

If~$p=2$ then~$n=2^rm$ with~$m$ odd.
Thus
\[
(5^n-1)\frac{B_{2n}}{2n}
\equiv
(5^k-1)\frac{B_{2k}}{2k}\pmod{2^r}
\]
by Lemma~\ref{lemmaFiveCongruence1}.
For~$r>1$ this gives
\begin{equation}\label{equationElaineDickinson}
(5^n-1)\frac{t_{n}}{b_n}
\equiv
(5^k-1)\frac{t_{k}}{b_k}\pmod{2^r}
\end{equation}
and for~$r=1$ this gives
\[
(5^n-1)\frac{t_{n}}{b_n}
\equiv
-(5^k-1)\frac{t_{k}}{b_k}
\equiv
(5^k-1)\frac{t_{k}}{b_k}\pmod{2},
\]
so~\eqref{equationElaineDickinson} holds for all~$r\ge1$.
By~\eqref{equationbsequenceasproduct} there are
odd positive integers~$g_n,h$
with~$h\equiv1$ modulo~$2^r$ such that~$b_n=2^{r+2}g_nh$
and~$b_k=2^{r+1}g_n$.
Thus~\eqref{equationElaineDickinson} can be written
\[
\Bigl(\frac{5^n-1}{2^{r+2}}\Bigr)\frac{t_n}{g_nh}
\equiv
\Bigl(\frac{5^k-1}{2^{r+1}}\Bigr)\frac{t_k}{g_n}
\pmod{2^r}
\]
so
\[
\frac{t_n}{g_nh}
\equiv
\frac{t_k}{g_n}
\pmod{2^r}
\]
by Lemma~\ref{lemmaStayingAlive}
and hence~$t_n\equiv ht_k$
modulo~$2^r$ since~$g_n$ and~$d_n$ are odd.
This gives~$t_n\equiv t_k$ modulo~$2^r$
since~$h\equiv1$ modulo~$2^r$.

Now assume that~$p$ is odd so we have~$2n=p^r(p-1)m$
for some~$m\ge1$ with~$p\notdivides m$.
By~\eqref{equationYoungCongruence}
we have
\[
(\gamma_p^{2n}-1)\frac{B_{2n}}{2n}
\equiv
(\gamma_p^{2k}-1)\frac{B_{2k}}{2k}\pmod{p^r}
\]
where~$\gamma_p>1$ is a primitive root modulo~$p$
with~$p^2\notdivides\gamma_p^{p-1}-1$.
By~\eqref{equationbsequenceasproduct} we can write~$b_n=p^{r+1}g_nh$
and~$b_k=p^rg_n$ with~$p\notdivides g_n$ and~$h\equiv1$
modulo~$p^r$ and so
\[
\Bigl(\frac{\gamma_p^{2n}-1}{p^{r+1}}\Bigr)
\frac{t_n}{g_nh}
\equiv
\Bigl(\frac{\gamma_p^{2k}-1}{p^{r}}\Bigr)
\frac{t_k}{g_n}\pmod{p^r}.
\]
By Lemma~\ref{lemmaPrimitiveRootIdea1} this
shows that
\[
t_n
\equiv
ht_k
\equiv
t_k
\pmod{p^r},
\]
so~$p^r\divides o_n$.

Finally, for~$n>1$ we write as usual~$n=p^rm$
with~$p\notdivides m$,~$r\ge1$ and some prime~$p$ dividing~$n$.
If~$k=\frac{n}{p}$ the arguments above
show that in all cases~$p^r\divides t_n-t_k$, so
the identity
\[
o_n
=
\sum_{d\smalldivides m}\mu\bigl(\tfrac{m}{d}\bigr)\bigl(t_{p^rd}-t_{p^{r-1}d}\bigr)
\]
shows that~$p^r\divides o_n$ and hence~$n\divides o_n$.
\end{proof}

\subsection{Sign Condition for Bernoulli Numerators}

In order to check the sign condition for the numerator
sequence~$t$ we will need the fact that it is preserved
under products (independently of whether the Dold congruence
holds, and without assuming the terms to be integers).
Recall that a real sequence~$\alpha$ is said to satisfy the
sign condition if~$\sum_{d\smalldivides n}\mu\bigl(\frac{n}{d}\bigr)\alpha_d\ge0$
for all~$n\ge1$.

\begin{lemma}\label{lemmaSignPreservedProducts1}
Assume that~$\alpha^{(1)}$ and~$\alpha^{(2)}$ are non-negative real sequences
satisfying the sign condition.
Then the product sequence~$\alpha^{(1)}\alpha^{(2)}=(\alpha_n^{(1)}\alpha_n^{(2)})$
also satisfies the sign condition.
\end{lemma}

\begin{proof}
By density it is sufficient to show this for rational
sequences, so we assume that~$\alpha^{(j)}_n$
is rational for all~$n\ge1$
and~$j=1,2$.
Fix some integer~$m\ge1$ and consider the finite
sequences~$(\alpha^{(j)}_n)_{n=1,\dots,m}$ for~$j=1,2$.
These satisfy the sign condition at each~$n\le m$ but
have no reason to satisfy the Dold congruence.
Write~$\denom(r)$ for the denominator of~$r\in\QQ$
written in lowest terms, and define
\[
C=\lcm\Bigl\{
\denom\Bigl(\frac{1}{n}\sum_{d\smalldivides n}\mu\bigl(\tfrac{n}{d}\bigr)\alpha^{(j)}_d\Bigr)
\mid
n=1,\dots,m\text{ and }j=1,2
\Bigr\}.
\]
This is a `repair factor' in the sense of
Miska and Ward~\cite{MR4361581}, meaning that~$(C\alpha^{(j)}_n)_{n=1,\dots,m}$
are finite sequences satisfying the Dold congruence at each point for~$j=1,2$.
Define
\[
\beta_n^{(j)}
=
\begin{cases}
\sum_{d\smalldivides n}\mu\bigl(\tfrac{n}{d}\bigr)C\alpha^{(j)}_d&\text{if }1\le n\le m\\
n&\text{if }n>m
\end{cases}
\]
and set~$\overline{\alpha}_n^{(j)}=\sum_{d\smalldivides n}\beta_d$
for all~$n\ge1$ and~$j=1,2$.
Then~$\overline{\alpha}^{(j)}$ is a sequence
satisfying both the Dold congruence and the sign
condition and hence is realizable for~$j=1,2$.
The product of two realizable sequences is realizable
simply by taking the Cartesian product of
the realizing systems, so~$\overline{\alpha}^{(1)}\overline{\alpha}^{(2)}$
is a realizable sequence.
In particular, this sequence satisfies the sign condition
for~$1\le n\le m$. By construction we have~$
\overline{\alpha}^{(1)}_n\overline{\alpha}^{(2)}_n
=
C^2\alpha^{(1)}_n\alpha^{(2)}_n
$
for~$1\le n\le m$, so the termwise product sequence~$\alpha^{(1)}\alpha^{(2)}$
satisfies the sign condition in the same range.
Since~$m\ge1$ was arbitrary, this proves the lemma.
\end{proof}

\begin{lemma}\label{lemmaNumeratorSign}
For any~$n\ge1$ we have~$o_n\ge0$.
\end{lemma}

\begin{proof}
The well-known formula relating values of the Riemann
zeta function at even integers to Bernoulli numbers
may be written as
\begin{equation}\label{equationBaselProblem1}
\alpha'_n
=
\frac{t_n}{b_n}
=
\frac{2(2n-1)!}{(2\pi)^{2n}}\zeta(2n)
\end{equation}
for all~$n\ge1$.
An induction argument using~\eqref{equationBaselProblem1}
shows that~$\alpha'_{n+1}\ge\alpha'_n$
and~$\alpha'_{2n}\ge n\alpha'_n$ for~$n>2$.
Define a rational sequence~$\alpha$
by setting~$\alpha_n=0$ for~$n\le2$
and~$\alpha_n=\alpha'_n$ for~$n\ge3$
so~$\alpha$ is a non-decreasing sequence
of non-negative rationals with~$\alpha_{2n}\ge n\alpha_{n}$
for all~$n\ge1$.
Then
\[
\alpha_{2n+1}
\ge
\alpha_{2n}
\ge
n\alpha_n
\]
for~$n\ge1$, so~$\alpha_n\ge\lfloor\tfrac{n}{2}\rfloor
\alpha_{\lfloor{n}/{2}\rfloor}$ for~$n\ge2$.
It follows that
\begin{equation}\label{equationAlphaHasSignCondition}
\sum_{d\smalldivides n}\mu\bigl(\tfrac{n}{d}\bigr)\alpha_d
\ge
\alpha_n
-
\sum_{d\smalldivides n\atop{d<n}}\alpha_d
\ge
\alpha_n-\lfloor{n}/{2}\rfloor\alpha_{\lfloor{n}/{2}\rfloor}
\ge
0
\end{equation}
for~$n\ge2$ since~$\mu(n)\ge-1$ for all~$n\in\NN$.

By Theorem~\ref{theoremBTEresult}(a) the
denominator sequence~$b$ is realizable,
so
\begin{equation}\label{equationSignForbSequence}
\sum_{d\smalldivides n}\mu\bigl(\frac{n}{d}\bigr)b_d\ge0
\end{equation}
for all~$n\ge1$.
By Lemma~\ref{lemmaSignPreservedProducts1} it follows
that the sequence~$b\alpha=(b_n\alpha_n)$
satisfies the sign condition.
We have~$b_n\alpha_n=t_n$ for~$n\ge3$
and~$b_n\alpha_n=0$ for~$1\le n\le2$.
If~$t$ does not satisfy the sign condition
then there is some~$n>1$ that bears witness
to this in the sense that~$w_n=\sum_{d\smalldivides n}\mu\bigl(\frac{n}{d}\bigr)t_d<0$.
Every term of~$t$ is odd so~$w_n$ is even,
and since~$t$ agrees with the realizable
sequence~$b\alpha$ apart from the first two
terms we must have~$w_n=-2$.
This is only possible if~$\mu(n)=\mu(n/2)=-1$,
which is impossible.
\end{proof}

Lemmas~\ref{lemmaNumeratorCongruence} (the Dold congruence)
and~\ref{lemmaNumeratorSign} (the sign condition)
together prove that the Bernoulli numerator sequence~$t$
is realizable, proving the first statement
in Theorem~\ref{theoremBTEresult}(b).

\subsection{Local Structure of Bernoulli Numerators}

We now turn to the local structure of the numerator
sequence~$t$. We have~$\lfloor t_n\rfloor_{p}=1$ for all~$n\ge1$ and~$p=2$ or~$p=3$
so the notion of (Bernoulli) regular prime for $2$ and~$3$
need not concern us.

\begin{lemma}\label{lemmaBernoulliRegularPrimesNumerators}
The numerator sequence~$t$ is realizable
at every regular prime.
\end{lemma}

\begin{proof}
In fact we will see that this holds for a trivial
reason:~$\lfloor t_n\rfloor_p=1$ for
all~$n\ge1$ and all regular primes.
This is clear for~$p\le3$ so assume that~$p\ge5$
is regular. Assume that~$t$ is not trivial
at the prime~$p$, and
choose~$n\ge1$ to be the smallest~$n$
with~$p\divides t_n$.
Then~$p\divides\frac{B_{2n}}{2n}$
so~$2n>p-3$ (by definition of regularity)
and hence~$p-1\le2n$.
If~$p-1\divides2n$ then~$p\divides b_n$
by~\eqref{equationbsequenceasproduct},
which contradicts the assumption that~$p\divides t_n$.
So we must have~$p-1\notdivides2n$
and hence~$m=\frac{2n-p+1}{2}\ge1$.
By Kummer's theorem this gives~$\frac{B_{2n}}{2n}\equiv\frac{B_{2m}}{2m}$
modulo~$p$ and so~$p\divides\frac{B_{2m}}{2m}$ and~$p\divides t_m$,
contradicting the minimality of~$n$.
\end{proof}

\begin{lemma}\label{lemmaBernoulliIrregularPrimesNumertators}
The numerator sequence~$t$ is not realizable at any irregular prime.
\end{lemma}

\begin{proof}
Assume that~$p$ is an irregular prime.
There is
by definition some~$k\ge1$ with
\[
1\le k\le\frac{p-3}{2}
\]
and~$p\divides B_{2k}$. This implies that~$p\divides t_k$.
Let~$m=\frac{k(p-1)}{2}$ so~$k\divides m$.
Since~$p-1\divides 2m$ von Staudt--Clausen shows
that~$p\divides b_m$ and hence~$p\notdivides t_m$
so~$\lfloor t_k\rfloor_p>\lfloor t_m\rfloor_p$.
If the sequence~$\lfloor t\rfloor_p$
is realizable then we have~$\lfloor t_k\rfloor_p\le\lfloor t_m\rfloor_p$
since~$k\divides m$.
It follows that~$\lfloor t\rfloor_p$ cannot be realizable.
\end{proof}

Lemma~\ref{lemmaBernoulliRegularPrimesNumerators} shows that~$t$
is (trivially) realizable at every regular prime,
and Lemma~\ref{lemmaBernoulliIrregularPrimesNumertators}
shows that~$t$ is not realizable at every
irregular prime, completing the proof of
Theorem~\ref{theoremBTEresult}(b).

\section{The Euler Sequence}

We have a less complete picture for the Euler
numbers. Arias de Reyna~\cite{zbMATH02203410}
showed how to use the Kummer congruences to
prove that~$e$ is realizable; we include a short
proof for completeness.
As with the Bernoulli denominators and numerators,
we first assemble some congruences of Kummer type.
By a result of Wagstaff~\cite{MR1956285}
we know that
\begin{equation}\label{equationWagstaff1}
2^{2n+1}A_{2n}\Bigl(\frac{p-1}{2}\Bigr)
=
\sum_{k=0}^{2n}\binom{2n}{k}E_{k}p^{2n-k}
\end{equation}
where~$A_n(m)=\sum_{k=1}^{m}(-1)^{m-k}k^n$
for~$m\ge1$ and~$n\ge0$.
Assume now that~$p$ is an odd prime
and~$m,n,r\ge1$ with~$m\ge n$ satisfying~$2m\equiv2n$
modulo~$\phi(p^r)$.
Write~$2m=2n+s\phi(p^r)$ with~$s\ge0$
so~$2^{2m+1}\equiv2^{2n+1}$
modulo~$p^r$ by Euler--Fermat.
If~$1\le k\le c<p$ then Euler--Fermat
also gives~$k^{2m}\equiv k^{2n}$
modulo~$p^r$.
It follows that
for~$1\le c\le p-1$ we have
\begin{equation}\label{equationACongruence1}
2^{2m+1}A_{2m}(c)
\equiv
2^{2n+1}A_{2n}(c)\pmod{p^r}.
\end{equation}

\begin{lemma}\label{lemmaEuler0}
If~$m,n,r\ge1$ satisfy~$1\le r\le2n+1\le2m+1$
and~$p\in\PP$ is odd with~$2m\equiv2n$ modulo~$\phi(p^r)$
then~$E_{2m}\equiv E_{2n}$ modulo~$p^r$.
\end{lemma}

\begin{proof}
By~\eqref{equationWagstaff1} we have
\begin{align*}
2^{2m+1}A_{2m}(c)
&=
\sum_{k=0}^{2m}\binom{2m}{k}E_{k}p^{2m-k}\\
\intertext{and}
2^{2n+1}A_{2n}(c)
&=
\sum_{k=0}^{2n}\binom{2n}{k}E_{k}p^{2n-k}
\end{align*}
where~$c=\frac{p-1}{2}$
so
\[
\sum_{k=0}^{2m}\binom{2m}{k}E_kp^{2m-k}
\equiv
\sum_{k=0}^{2n}\binom{2n}{k}E_kp^{2n-k}
\pmod{p^r}
\]
by~\eqref{equationACongruence1}.
This shows that
\begin{equation}\label{equationUseful1Euler}
\sum_{k=0}^{r-1}\binom{2m}{k}E_{2m-k}p^k
\equiv
\sum_{k=0}^{r-1}\binom{2n}{k}E_{2n-k}p^k\pmod{p^r}.
\end{equation}
by adjusting the range
of summation and reducing modulo~$p^r$.

For~$r=1$ this shows that
\begin{equation}\label{equationr=1}
E_{2m}
\equiv
E_{2n}\pmod{p}.
\end{equation}
For~$r>1$ assume that~$1\le s\le 2n+1\le 2m+1$ and
that
\begin{equation}\label{equationInductiveStep1}
2m\equiv2n\pmod{p^{s}}
\Longrightarrow
E_{2m}\equiv E_{2n}\mod{p^s}
\end{equation}
for~$1\le s<r$,
and notice that the latter congruence
implies~$E_{2m}\equiv E_{2n}$
modulo~$p^{r-1}$ since~$\phi(p^s)\divides\phi(p^r)$
for~$1\le s<r$.
Now by~\eqref{equationUseful1Euler}
this gives
\[
E_{2m}+\sum_{k=1}^{r-1}\binom{2n}{k}\bigl(E_{2m-k}\!-\!E_{2n-k}\bigr)p^k
+
\sum_{k=1}^{r-1}\Bigl(\binom{2m}{k}\!-\!\binom{2n}{k}\Bigr)E_{2m-k}p^k
\equiv
E_{2n}
\]
modulo~$p^r$
and so by the assumption~\eqref{equationInductiveStep1} this
reduces to
\[
E_{2m}+\sum_{k=1}^{r-1}\Bigl(\binom{2m}{k}-\binom{2n}{k}\Bigr)E_{2m-k}p^k
\equiv
E_{2n}\pmod{p^r}.
\]
On the other hand~$\binom{2m}{k}\equiv\binom{2n}{k}$ modulo~$p^{r-k}$
so we deduce that~$E_{2m}\equiv E_{2n}$ modulo~$p^r$.
\end{proof}

It follows that if~$p\in\PP$ is odd
and~$b,r\ge1$ have~$p\notdivides b$ then we
have
\[
1\le r\le2p^{r-1}b-1<2p^rb-1
\]
and so
\begin{equation}\label{equationAdditiveCongruence1}
E_{2p^rb}\equiv E_{2p^{r-1}b}\pmod{p^r}
\end{equation}
since~$2p^rb\equiv2p^{r-1}b$ modulo~$\phi(p^r)$.
The congruence~\eqref{equationAdditiveCongruence1}
extends to the prime~$p=2$ as follows.
Wagstaff~\cite{MR1956285} shows that
\begin{equation*}\label{equationAdditiveCongruence2}
E_{2n}\equiv E_{2n+2^k}+2^k\pmod{2^{k+1}}
\end{equation*}
for~$k,n\ge1$ and so
\begin{equation}\label{equationAdditiveCongruence3}
E_{2^{r+1}b}\equiv E_{2^rb}\pmod{2^{r}}
\end{equation}
for~$b,r\ge1$ with~$b$ odd.
Finally~\eqref{equationAdditiveCongruence1}
and~\eqref{equationAdditiveCongruence3}
together show that if~$p\in\PP$
and~$m,r\ge1$ with~$p\notdivides m$
we have
\begin{equation}\label{equationAdditiveCongruence4}
E_{2p^{r}m}\equiv E_{2p^{r-1}m}\pmod{p^{r}}.
\end{equation}

\subsection{Realizability of the Euler Sequence}

As with the Bernoulli numerator sequence, we approach this
by directly verifying the Dold congruence and the sign
condition. For brevity we define an integer
sequence by
\begin{equation}\label{equationOrbitsForEulerSequence}
o^{(e)}_n=\sum_{d\smalldivides n}\mu\bigl(\tfrac{n}{d}\bigr)e_d
\end{equation}
so~$o^{(e)}=(1, 0, 4, 60, 1384, 50516, 2702764,\dots)$;
notice that~$(\frac{1}{n}e^{(e)}_n)=\seqnum{A060164}$.

\begin{lemma}\label{lemmaDoldForEulerSequence1}
The Euler sequence~$e$ satisfies the Dold congruence.
\end{lemma}

\begin{proof}
This is clear for~$n=1$ so assume that~$n>1$
and let~$n=p^rm$ with~$p\in\PP$,~$r\ge1$,
and~$p\notdivides m$ so that~$o^{(e)}_n
=\sum_{d\smalldivides m}\mu\bigl(\frac{m}{d}\bigr)\bigl(e_{p^rd}-e_{p^{r-1}d}\bigr)$.
By definition we have
\begin{align*}
e_{p^rd}-e_{p^{r-1}d}
&=
(-1)^{p^{r-1}d}\bigl(
(-1)^{p^{r-1}(p-1)d}
E_{2p^rd}-E_{2p^{r-1}d}
\bigr).
\end{align*}
If~$p$ is odd and~$r\ge1$ or if~$p=2$
and~$r>1$ we therefore have
\[
e_{p^rd}-e_{p^{r-1}d}
=
(-1)^{p^{r-1}d}\bigl(
E_{2p^rd}-E_{2p^{r-1}d}
\bigr).
\]
For~$p=2$ and~$r=1$ we have
\[
e_{2d}-e_d
=
(E_{4d}-E_{2d})+2E_{2d}
\equiv
E_{4d}-E_{2d}
\pmod{2}.
\]
By~\eqref{equationAdditiveCongruence4}
it follows that~$p^r\divides o_n^{(e)}$ and
hence~$n\divides o_n^{(e)}$ as desired.
\end{proof}

\begin{lemma}\label{lemmaEulerSignCondition1}
The Euler sequence~$e$ satisfies the sign condition.
\end{lemma}

\begin{proof}
It is easy to show that~$e$ satisfies~$e_{2n}\ge ne_n$
for all~$n\ge1$. For example, this is a consequence of
the formula
\[
\vert E_{2n}\vert
=
\frac{2^{2n+2}(2n)!}{\pi^{n+1}}\Bigl(1-\frac{1}{3^{2n+1}}+\frac{1}{5^{2n+1}}-\frac{1}{7^{2n+1}}+\cdots\Bigr).
\]
Just as in the proof of Lemma~\ref{lemmaNumeratorSign},
this shows that~$e$ satisfies the sign condition.
\end{proof}

Lemmas~\ref{lemmaDoldForEulerSequence1}
and~\ref{lemmaEulerSignCondition1} together
show that~$e$ is realizable, proving the first statement in
Theorem~\ref{theoremBTEresult}(c).

The second statement in Theorem~\ref{theoremBTEresult}(c)
simply involves finding a prime that witnesses the
failure of local realizability:
The~$61$-part of~$e$ begins
\[
\lfloor e\rfloor_{61}
=
(1,1,61,1,1,1,1,1,1,1,\dots)
\]
which fails both the Dold congruence
and the sign condition at~$n=9$
since~$\sum_{d\smalldivides9}\mu\bigl(\frac{9}{d}\bigr)\lfloor e_d\rfloor_{61}=-60$.

Thus, in particular,~$e$ cannot be nilpotently realizable by
Theorem~\ref{theoremNilpotent} and in our setting this
may be viewed as an explanation for its relative
intractability relative to the Bernoulli denominator
sequence~$b$.

\subsection{Euler Regular Primes}

Calculations suggest the following more precise version of
Theorem~\ref{theoremBTEresult}(c).
A prime~$q$ is \emph{Euler irregular}
if there exists some~$n$ with~$0<n<\frac{q-1}{2}$
such that~$q$
divides~$e_n=|E_{2n}|$.
A prime~$q\in\PP$ is \emph{Euler regular}
if it is not irregular, so~$q\notdivides e_n$ for
all~$n$ with~$0<n<\frac{q-1}{2}$.
The sequence of Euler irregular primes begins
\[
(19,31,43,47,61,67,71,79,101,\dots)=\seqnum{A120337}.
\]
It is well-known that there are infinitely many
Euler irregular primes
(see Carlitz~\cite{zbMATH03092654};
Ernvall~\cite{zbMATH03489197} showed further that there
are infinitely many congruent to~$1$ modulo~$8$), but only conjectured that there
are infinitely many Euler regular primes.
For our purposes it is helpful to further
divide the Euler regular primes as follows.

\begin{definition}
A prime~$q$ is \emph{strong Euler regular}
if~$q\notdivides e_n$ for all~$n\ge1$.
A prime~$q$ is \emph{weak Euler regular}
if it is Euler regular but not strong Euler regular.
\end{definition}

Thus the sequence of strong Euler
regular primes begins
\[
(2,3,7,11,23,59,83,103,\dots)=\seqnum{A092217}
\]
and the sequence of weak Euler
regular primes begins
\[
(5,13,17,29,37,41,53,\dots).
\]
Clearly~$q\in\PP$ is strong Euler regular
if and only if~$\lfloor e\rfloor_q$
is the trivial sequence~$(1)$.
Calculations suggest that~$q\in\PP$ is weak
Euler regular if and only if
\[
\lfloor e_n\rfloor_q
=
\begin{cases}
q^{1+\ord_q(n)}&\mbox{if }\frac{q-1}{2}\divides n;\\
1&\mbox{if not}.
\end{cases}
\]
This would imply that~$e$ is realizable
at~$q\in\PP$ if and only if~$q$ is Euler regular.

\section{Shift Invariance}

Arias de Reyna's beautiful argument
for the realizability of~$e$
in fact shows something
more about the Euler sequence~$e$: It remains realizable
after translation.
That is,~$(e_{n+k})$ is also realizable for any~$k\ge0$.
In order to see how surprising this is, notice
that the property of realizability is a condition
involving the values of the sequence
at the divisors of~$n$ and so is intimately connected
to the multiplicative structure of the semigroup~$\mathbb{N}$.
Addition in~$\mathbb{N}$---shifting the sequence---disrupts
the multiplicative structure in unpredictable ways,
so finding a non-trivial instance of a realizable
sequence with realizable shifts is more than unexpected.
To see how this looks for the Euler sequence in~\eqref{equationShiftingEulerSequence},
we may construct the closed orbits showing realizability
in the following sense:
The sequence~$e$ is realizable because
the associated map may be constructed
as having~$1$ fixed point,~$2$ closed orbits
of length~$2$,~$20$ closed orbits of length~$3$, and so on.
We may think of the map as the following combinatorial object,
with the sequence of closed
orbit counts~$(1,2,20,345,\dots)=\seqnum{A060164}$.
\[
\underbrace{\Ngon{1}}_{1\text{ copy}}
\sqcup
\underbrace{\Twogon}_{2\text{ copies}}
\sqcup
\underbrace{\Ngon{3}}_{20\text{ copies}}
\sqcup
\underbrace{\Ngon{4}}_{345\text{ copies}}
\sqcup
\underbrace{\Ngon{5}}_{10104\text{ copies}}
\sqcup
\cdots
\]
Shifting the sequence by one gives~$(e_{n+1})
=(5,61,1385,50521,2702765,\dots)$ which is realizable
because of some other map which may be
thought of as follows:
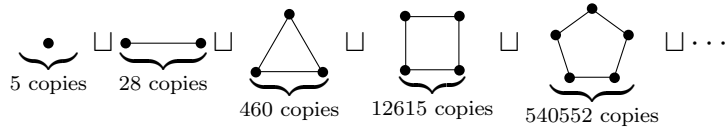
\[
\underbrace{\Ngon{1}}_{5\text{ copies}}
\sqcup
\underbrace{\Twogon}_{28\text{ copies}}
\sqcup
\underbrace{\Ngon{3}}_{460\text{ copies}}
\sqcup
\underbrace{\Ngon{4}}_{12615\text{ copies}}
\sqcup
\underbrace{\Ngon{5}}_{540552\text{ copies}}
\sqcup
\cdots
\]
corresponding to the sequence of orbit
counts~$(5,28,460,12615,\dots)$.

The combinatorial complexity of the
shift map viewed on orbit counts,
in this instance sending~$(1,2,20,345,\dots)$
to~$(5,28,460,12615,\dots)$ motivates the following extravagant terminology.

\begin{definition}
A sequence~$a=(a_n)$ of
non-negative integers is called \emph{magical}
if~$(a_{n+k})$ is realizable for all~$k\ge0$.
\end{definition}

The Bernoulli denominator sequence~$b$ behaves rather differently.
By the work of Arias de Reyna its shifts
are `pre-realizable', meaning that they satisfy the
Dold congruence. In fact this result is sharp in the
sense that~$(b_{n+1})$
is not realizable because
\[
\sum_{d\divides4}\mu(\tfrac4d)b_{d+1}=-120
\]
is not non-negative.
The shifts of the Bernoulli numerator sequence~$t$
are not even pre-realizable since
\[
\tfrac{1}{7}\sum_{d\divides7}\mu(\tfrac4d)t_{d+1}
=
\tfrac{3616}{7}
\]
is not an integer.

\begin{example}\label{exampleFullTwoShift}
The full shift on two symbols (or, in number
theoretic language, Euler--Fermat) shows
that~$a=(2^n)_{n\ge1}=
\seqnum{A000079}$
is realizable.
By taking~$2^k$ disjoint copies of the
map we see that~$(2^{n+k})$ is
realizable for any~$k\ge0$
and so~$a$ is magical.
\end{example}

Thus it is possible for a linear recurrence sequence
to be magical.

\begin{example}\label{exampleGoldenMeanShift}
The so-called golden mean shift shows that
the Lucas sequence
\[
(1,3,4,7,\dots)=\seqnum{A000032}
\]
is realizable. However shifting this
sequence by one gives the
sequence~$(3,4,7,\dots)$
which is clearly not realizable
(the corresponding dynamical system
would need to have~$3$ fixed points
and hence half an orbit of length~$2$).
\end{example}

\begin{example}
The sequence~$(2^n-1)$ is magical.
That this is somewhat deceptive may
be seen in two different ways:
In dynamics, it is really best thought
of as the difference of
Example~\ref{exampleFullTwoShift} and
the identity map on a set with one element,
each of which is clearly magical---and the
magical property is preserved under addition
and subtraction of sequences.
As a linear recurrence, this is the
difference of two trace sequences.
\end{example}

For other algebraic realizable sequences we
should not expect this property. Example~\ref{exampledihedral}
shifted by one begins~$(4,4,8,\dots)$ and so is not
realizable; Example~\ref{exampletoral} shifted by one
begins~$(1,1,5,\dots)$ and so is also
not realizable.
Mateusz Rajs~\cite{rajs} has completely classified
the magical linear recurrence sequences, showing that
the examples above are in a sense representative of
the phenomena at work.
It would be
interesting to gain more insight into magical sequences
in general.

\section{Further Local Questions}

It is straightforward to numerically test for local
realizability and we simply list without detailed
proof some provocative
examples of what these calculations suggest for a small
number of well-known sequences. In the list below
we write `realizable* at' and `realizable{\ddag}' to signify that the
claimed statements
have not been proved but simply numerically observed
using the number of terms available in the
Online Encyclopedia of Integer Sequences.
Specifically, a `realizable* at' sequence
of primes below
might have erroneously \emph{included} some primes.
The
sequence of primes at which the sequence
is not realizable is the complement of
the sequence of primes at which the sequence
is realizable. In each case
this means there is a point in the localized
sequence that bears witness to this fact, but
this being checked in a finite calculation
we write `not realizable{\textdagger}' to record
the fact that the list might have
erroneously \emph{excluded} some primes
whose witnesses lay beyond the calculations.
Indeed, in each case we have no {\it{a priori}}
reason to believe that the sequence of primes at
which the original sequence is realizable or
not realizable is infinite.
With these caveats understood,
the presence of primes at which
the sequences are not realizable shows by Theorem~\ref{theoremNilpotent} that
all of the sequences below
are not nilpotently algebraically realizable.
The analogy with Theorem~\ref{theoremBTEresult} suggests that
in each case this gives a compelling notion of
`regular prime' for that particular sequence.
A further calculation shows that the only magical
sequence in the list below is the sequence of
tangent numbers, illustrating again how unexpected it
is to find non-trivial instances of that property.
\begin{enumerate}
\item The Lucas sequence (which is realizable but not algebraically realizable)
may be defined by the linear recurrence~$a_{n+2}=a_{n+1}+a_n$ with
initial conditions~$a_1=1$ and~$a_2=3$. It begins
\[
(1, 3, 4, 7, 11, 18, 29, 47, 76, 123, 199, 322, 521,\dots)
=
\seqnum{A000032},
\]
is realizable* at
\[
(5, 11, 13, 17, 19, 29, 31, 37, 41, 53, 59, 61, 71, 73, 79, 83, 89, 97,\dots),
\]
and is not realizable{\textdagger} at
\[
(2, 3, 7, 23, 43, 47, 67, 107,\dots).
\]
\item The sequence of Domb numbers
with~$n$th term defined by
\[
\sum_{k=0}^{n}\binom{n}{k}^2\binom{2n-2k}{n-k}\binom{2k}{k}
\]
begins
\[
(4, 28, 256, 2716, 31504, 387136, 4951552,\dots)
=
\seqnum{A002895}.
\]
It is realizable by a result of Zhang~\cite{MR4722519},
is realizable* at
\[
(3, 5, 13, 17, 29, 31, 37, 41, 43, 47, 53, 59, 61, 67, 71, 73, 79,\dots),
\]
and is not realizable{\textdagger} at
\[
(2, 7, 11, 19, 23, 179,\dots).
\]
\item The sequence of Ap{\'e}ry numbers of the first kind
has~$n$th term given by
\[
\sum_{k=0}^{n}\Bigl(\binom{n}{k}\binom{n+k}{k}\Bigr)^2.
\]
It begins
\[
(1, 5, 73, 1445, 33001, 819005, 21460825,\dots)
=
\seqnum{A005259}
\]
and is realizable by Zhang~\cite{MR4722519}.
It is realizable* at
\[
(2, 3, 7, 13, 23, 29, 37, 43, 47, 53, 61, 67, 71, 79, 83, 89, 97, 101,\dots)
\]
and is not realizable{\textdagger} at
\[
(5, 11, 17, 19, 31, 41, 59, 73,\dots).
\]
\item The sequence of Ap{\'e}ry numbers of the second kind
with~$n$th term given by
\[
\sum_{k=0}^{n}\binom{n}{k}^2\binom{n+k}{k}
\]
begins
\[
(1, 3, 19, 147, 1251, 11253, 104959, 1004307,\dots)
=
\seqnum{A005258}.
\]
It is realizable by Zhang~\cite{MR4722519},
is realizable* at
\[
(2, 5, 13, 17, 23, 29, 37, 41, 43, 47, 53, 59, 61, 67, 73, 79, 89,\dots),
\]
and is not realizable{\textdagger} at
\[
(3, 7, 11, 19, 31, 71, 83, 139, 157,\dots).
\]
\item The sequence of quadrinomial coefficients
with~$n$th term equal to the coefficient of~$x^n$
in~$(1+x+x^2+x^3)^n$ begins
\[
(1, 3, 10, 31, 101, 336, 1128, 3823, 13051, 44803,\dots)
=
\seqnum{A005725}.
\]
It is realizable by Zhang~\cite{MR4722519},
is realizable* at
\[
(23, 41, 61, 73, 79, 83, 89, 97, 103, 107, 109, 113,\dots),
\]
and is not realizable{\textdagger} at
\[
(2, 3, 5, 7, 11, 13, 17, 19, 29, 31, 37, 43, 47, 53, 59, 67,\dots).
\]
\item The Fibonacci sequence sampled along the squares
and multiplied by~$5$,
\[
(5, 15, 170, 4935, 375125, 74651760, 38893710245, \dots)
=
5\times\seqnum{A054783}
\]
is realizable by Moss and Ward~\cite{MR4394356}.
It is realizable*
at
\[
(3, 5, 11, 19, 29, 31, 37, 41, 43, 59, 61, 67, 71, 73, 79,\dots)
\]
and not realizable{\textdagger} at
\[
(2, 7, 13, 17, 23, 47, 53, 97, 107,\dots).
\]
\item The Catalan--Larcombe--French sequence
(whose terms arise
as coefficients in the series expansion of an elliptic integral of the first kind)
begins
\[
(8, 80, 896, 10816, 137728, 1823744, 24862720, \dots)
=
\seqnum{A053175}.
\]
It may be shown to be realizable using known
congruences from work of Ji and Sun~\cite{MR3666639}.
It is realizable* at
\[
(3, 11, 17, 19, 43, 59, 73, 83, 89,\dots)
\]
and not realizable{\textdagger} at
\[
(2, 5, 7, 13, 23, 29, 31, 37, 41, 47, 53, 61, 67, 71, 79, 97,\dots).
\]
\item The~$(n-1)$th `tangent' number is~$2^{2n}\bigl(2^{2n}-1\bigr)
\frac{\vert B_{2n}\vert}{2n}$, giving the sequence
\[
(2, 16, 272, 7936, 353792, 22368256, 1903757312,\dots)
=
\seqnum{A000182}.
\]
As pointed out by Arias de Reyna this
is realizable and (easily shown to be) magical.
It is realizable* at
\[
(3, 5, 7, 11, 13, 19, 23, 29, 47, 53, 59, 61, 67, 71, 79, 83, 97,\dots)
\]
and not realizable{\textdagger} at
\[
(2, 17, 31, 37, 41, 43, 73, 89,\dots).
\]
\end{enumerate}

Finally, we record a remarkable
infinite
family of sequences
arising in path combinatorics. Let~$a_n^{(k)}$
denote the number of paths from~$(n,\dots,n)$
to~$(0,\dots,0)$ in~$\mathbb{Z}^k$ in which
each step decreases one or more components by~$1$, so
\[
a_n^{(k)}
=
\sum_{j=0}^{kn}\sum_{i=0}^{j}(-1)^j\binom{j}{i}\binom{j-i}{n}^k
\]
for all~$n\ge0$ and~$k\ge2$.
For example,~$a^{(2)}=(a_n^{(2)})$
is the sequence of central Delannoy numbers
\[
(1, 3, 13, 63, 321, 1683, 8989, 48639, 265729, 1462563,\dots)
=
\seqnum{A001850}.
\]
Using known congruences for the binomial coefficients
it is possible to show that each of these
sequences~$a^{(k)}$ is realizable,
but calculations show they are not
realizable at every prime
as follows:
\begin{enumerate}
\setcounter{enumi}{8}
\item For~$k=2$ the sequence~$a^{(2)}$
is realizable* at
\[
(2, 5, 29, 37, 41, 59, 61, 67, 73, 83,\dots)
\]
and not realizable{\textdagger} at
\[
(3, 7, 11, 13, 17, 19, 23, 31, 43, 47, 53, 71, 79, 89, 97,\dots).
\]
\item For~$k=3$ we have
\[
a^{(3)}
=
(1, 13, 409, 16081, 699121, 32193253, 1538743249,\dots)
=
\seqnum{A126086}.
\]
It is realizable* at
\[
(2, 3, 5, 7, 17, 19, 23, 29, 37, 43, 47, 53, 67, 71, 83, 89, 97,\dots)
\]
and not realizable{\textdagger} at
\[
(11, 13, 31, 41, 59, 61, 73, 79,\dots).
\]
\item For~$k=4$ we have
\[
a^{(4)}
=
(1, 75, 23917, 10681263, 5552351121,\dots)
=
\seqnum{A263064}.
\]
It is realizable* at
\[
(2, 7, 17, 23, 29, 31, 37, 47, 53, 61, 67, 71, 73, 79, 83, 89, 97,\dots)
\]
and not realizable{\textdagger} at
\[
(3, 5, 11, 13, 19, 41, 43, 59,\dots).
\]
\item For~$k=5$ we have
\[
a^{(5)}
=
(1, 541, 2244361, 14638956721, 117029959485121,\dots)
=
\seqnum{A263065}.
\]
It is realizable* at
\[
(2, 3, 5, 11, 13, 17, 19, 23, 29, 31, 37, 41, 43, 47, 53, 59, 61, 67, 71, 73, \dots)
\]
and not realizable{\textdagger} at
\[
(7,\dots).
\]
\end{enumerate}

\section*{Acknowledgments}

This work is largely the product of discussions over some years
with Patrick Moss, and many of these ideas first appeared
in his thesis~\cite{pm}.
He passed away early in~2024~\cite{ward2024patrickmoss251019471732024}
and we dedicate this modest work to his memory.
We also thank Piotr Miska for telling us about
the work of Mateusz Rajs~\cite{rajs},
Grzegorz Graff for introducing the third
named author to Dorota Chanko and explaining her work,
and an anonymous referee for a particularly careful reading
and suggestions which improved the exposition.

This work was supported by the Fundamental
Fund 2026 of Khon Kaen University which has
received funding support from the National Science Research and Innovation Fund (NSRF)
of Thailand.

\section*{Data availability statement}

No data beyond what is described above and obtainable
from the OEIS~\cite{OEIS} has been used
in this manuscript.


\end{document}